\documentclass[letterpaper,10pt,reqno]{amsart}
\usepackage{indentfirst} 
\usepackage{amssymb}
\usepackage{mathtools} 
\usepackage{mathabx} 
\usepackage{amsthm}
\usepackage{thmtools}
\usepackage{enumitem} 
\usepackage[backref=page,colorlinks=true]{hyperref} 
\usepackage[usenames,dvipsnames]{xcolor} 
\usepackage{ifthen}
\usepackage{tikz}
\usetikzlibrary{decorations.pathreplacing}
\usepackage{tikz-cd}
  
\renewcommand*{\backref}[1]{}
\renewcommand*{\backrefalt}[4]{\tiny
  \ifcase #1 (\textbf{NOT CITED.})%
  \or    (Cited on page~#2.)%
  \else   (Cited on pages~#2.)%
  \fi}

\makeatletter

\def\MRbibitem{\@ifnextchar[\my@lbibitem\my@bibitem}

\def\mybiblabel#1#2{\@biblabel{{\hyperref{http://www.ams.org/mathscinet-getitem?mr=#1}{}{}{#2}}}}

\def\myhyperanchor#1{\Hy@raisedlink{\hyper@anchorstart{cite.#1}\hyper@anchorend}}

\def\my@lbibitem[#1]#2#3#4\par{%
  \item[\mybiblabel{#2}{#1}\myhyperanchor{#3}\hfill]#4%
  \@ifundefined{ifbackrefparscan}{}{\BR@backref{#3}}%
  \if@filesw{\let\protect\noexpand\immediate
    \write\@auxout{\string\bibcite{#3}{#1}}}\fi\ignorespaces%
}

\def\my@bibitem#1#2#3\par{%
  \refstepcounter\@listctr
  \item[\mybiblabel{#1}{\the\value\@listctr}\myhyperanchor{#2}\hfill]#3%
  \@ifundefined{ifbackrefparscan}{}{\BR@backref{#2}}%
  \if@filesw\immediate\write\@auxout
    {\string\bibcite{#2}{\the\value\@listctr}}\fi\ignorespaces%
}

\makeatother

\DeclareFontFamily{U} {MnSymbolA}{}
\DeclareFontShape{U}{MnSymbolA}{m}{n}{
   <-6> MnSymbolA5
   <6-7> MnSymbolA6
   <7-8> MnSymbolA7
   <8-9> MnSymbolA8
   <9-10> MnSymbolA9
   <10-12> MnSymbolA10
   <12-> MnSymbolA12}{}
\DeclareFontShape{U}{MnSymbolA}{b}{n}{
   <-6> MnSymbolA-Bold5
   <6-7> MnSymbolA-Bold6
   <7-8> MnSymbolA-Bold7
   <8-9> MnSymbolA-Bold8
   <9-10> MnSymbolA-Bold9
   <10-12> MnSymbolA-Bold10
   <12-> MnSymbolA-Bold12}{}
\DeclareSymbolFont{MnSyA} {U} {MnSymbolA}{m}{n}
 \DeclareFontFamily{U} {MnSymbolC}{}
\DeclareFontShape{U}{MnSymbolC}{m}{n}{
  <-6> MnSymbolC5
  <6-7> MnSymbolC6
  <7-8> MnSymbolC7
  <8-9> MnSymbolC8
  <9-10> MnSymbolC9
  <10-12> MnSymbolC10
  <12-> MnSymbolC12}{}
\DeclareFontShape{U}{MnSymbolC}{b}{n}{
  <-6> MnSymbolC-Bold5
  <6-7> MnSymbolC-Bold6
  <7-8> MnSymbolC-Bold7
  <8-9> MnSymbolC-Bold8
  <9-10> MnSymbolC-Bold9
  <10-12> MnSymbolC-Bold10
  <12-> MnSymbolC-Bold12}{}
\DeclareSymbolFont{MnSyC} {U} {MnSymbolC}{m}{n}

\DeclareMathSymbol{\top}{\mathord}{MnSyA}{219} 
\DeclareMathSymbol{\plus}{\mathord}{MnSyC}{20} 

\declaretheorem[numberwithin=section]{theorem}
\declaretheorem[sibling=theorem]{lemma}

\declaretheorem[sibling=theorem]{proposition}
\declaretheorem[sibling=theorem,style=definition]{definition}

\declaretheorem[sibling=theorem,style=remark]{remark}

\hypersetup{bookmarksdepth = 3} 
\numberwithin{equation}{section}     

\setlist[enumerate,1]{label={\upshape(\alph*)},ref=\alph*}
\setlist[enumerate,2]{label={\upshape(\arabic*)},ref=\arabic*}

\newcommand{\M}{\mathcal{M}}
\newcommand{\R}{\mathbb{R}}
\newcommand{\Z}{\mathbb{Z}}
\newcommand{\N}{\mathbb{N}}
\newcommand{\E}{\mathbb{E}}

\newcommand{\cE}{\mathcal{E}}

\def\phi{\varphi}
\def\R{{\mathbb R}}

\def\N{{\mathbb N}}
\def\Z{{\mathbb Z}}

\def\E{{\mathcal E}}

\def\M{{\mathcal M}}

\def\le{\leqslant}
\def\ge{\geqslant}

\def\M{\mathcal{M}}

\newcommand{\vertiii}[1]{{\left\vert\kern-0.25ex\left\vert\kern-0.25ex\left\vert #1 
    \right\vert\kern-0.25ex\right\vert\kern-0.25ex\right\vert}}
\newcommand{\invertiii}[1]{{\vert\kern-0.25ex\vert\kern-0.25ex\vert #1 
    \vert\kern-0.25ex\vert\kern-0.25ex\vert}}

\begin{document}

\title{Measures of maximal entropy for suspension flows}
\date{\today}


\begin{thanks}
{ G.I.\ was partially supported  by Proyecto Fondecyt 1190194 and by CONICYT PIA ACT172001}
\end{thanks}

\author[G.~Iommi]{Godofredo Iommi}
\address{Facultad de Matem\'aticas,
Pontificia Universidad Cat\'olica de Chile (PUC), Avenida Vicu\~na Mackenna 4860, Santiago, Chile}
\email{\href{giommi@mat.uc.cl}{giommi@mat.uc.cl}}
\urladdr{\url{http://http://www.mat.uc.cl/~giommi/}}

 \author[A.~Velozo]{Anibal Velozo}  \address{Department of Mathematics, Yale University, New Haven, CT 06511, USA.}
\email{\href{anibal.velozo@gmail.com}{anibal.velozo@yale.edu}}
\urladdr{\href{https://gauss.math.yale.edu/~av578/}{https://gauss.math.yale.edu/~av578/}}

\begin{abstract}
We study suspension flows defined over sub-shifts of finite type with continuous roof functions. We prove the existence of suspension flows with uncountably many ergodic measures of maximal entropy. More generally, we prove that any suspension flow defined over a sub-shift of finite type can be perturbed (by an arbitrarily small perturbation) so that the resulting flow has uncountably many ergodic measures of maximal entropy, and that the same can be arranged so that the new flow has a unique measure of maximal entropy. 
\end{abstract}

\maketitle

\section{Introduction}
In this article we study ergodic theory of suspension flows. These are continuous time dynamical systems defined over a base given by a discrete time system together with a roof function that determines the time the flow spends before returning to the base. We will consider the case when we have sub-shifts of finite type as base and continuous positive functions as roofs.

Suspension flows have been used as models for different classes of differentiable flows at least since the work of Hadamard \cite{h}, in which he studied closed geodesics in surfaces of negative curvature. In 1973, Bowen \cite{bo} and Ratner \cite{ra} constructed Markov partitions that 
allowed for the coding of Axiom A flows by suspension flows over sub-shifts of finite type with H\"older continuous roof functions. Later, Pollicott \cite{po} extended these constructions and was able to code Smale flows (which are a generalization of Axiom A flows) with suspension flows defined over sub-shifts of finite type with continuos roof functions. More recently, symbolic systems have been constructed for flows that are defined on non-compact surfaces or even having non-uniform forms of hyperbolicity. In these cases the base symbolic systems are countable Markov shifts (non-compact generalizations of sub-shifts of finite type), see for example \cite{ls}.

Thermodynamic formalism for suspension flows defined over sub-shifts of finite type with H\"older roof functions is well understood. There exists a unique measure of maximal entropy, the pressure is real analytic and H\"older functions have unique equilibrium measures (see Section \ref{pre} for definitions and \cite{br,pp} for details). In this article we show that the situation drastically changes if we only assume the roof function to be continuous. Indeed, we  prove the existence of a suspension flow with uncountably many ergodic measures of maximal entropy (see Proposition \ref{existence1}). 
In fact, we prove a much more general result: by an arbitrarily small time change, any suspension flow defined over a sub-shift of finite type can be perturbed so that the resulting flow has uncountably many ergodic measures of maximal entropy. We
also prove that the same can be arranged so that the new flow has a unique measure of maximal entropy. The main result of this article is the following,
\begin{theorem}\label{thm:1}
Let $(\Sigma, \sigma)$ be a topologically transitive one or two-sided sub-shift of finite type. Let $\tau: \Sigma \to \R$ be a positive continuous potential and $(Y_{\tau}, \Phi_{\tau})$ the corresponding suspension flow. Let $\epsilon>0$, then
\begin{enumerate}
\item\label{11} There exists a positive continuous function  $\tau_1: \Sigma \to \R$ with $\|\tau- \tau_1 \|_0<\epsilon$ such that the corresponding suspension flow  $(Y_{\tau_1}, \Phi_{\tau_1})$ has uncountably many ergodic measures of maximal entropy.
\item\label{12} There exists a positive continuous function  $\tau_2: \Sigma \to \R$ with $\|\tau- \tau_2\|_0<\epsilon$ such that the corresponding suspension flow  $(Y_{\tau_2}, \Phi_{\tau_2})$ has a unique measure of maximal entropy.
\end{enumerate}
\end{theorem}

Theorem \ref{thm:1} states that an arbitrarily small time reparametrization of  $(Y_\tau,\Phi_\tau)$ yields  systems  with uncountably many ergodic measures of maximal entropy and systems with a unique one. This result provides a description of the space of suspension flows defined over a fixed base. It shows that both,  flows with unique and those with uncountably many measures of maximal entropy form dense sets.


Our results are of  a complementary nature to those obtained by  Kucherenko and Thompson \cite[Theorem 2]{kt2} (see also \cite{kt1}). They consider suspension flows in which the base is a specific sub-shift of finite type, namely the full-shift on finitely many symbols.  With a constructive method they produce  roof functions for which the corresponding suspension flow has finitely many measures of maximal entropy. Moreover, they have control on the support of these measures. In a recent preprint, Kucherenko and Thompson \cite{kt3}  also constructed examples with  uncountably many measures of maximal entropy. 

Our strategy is different and only relies on results of Israel (Theorem \ref{is}) and Ruelle (Theorem \ref{ru}). Theorem \ref{is} says that among continuous functions, there is a dense set with uncountably many equilibrium states. We remark that this is a purely existencial result. In contrast, Theorem \ref{ru} establishes the uniqueness of equilibrium states for H\"older continuous functions.  In the proof of Theorem \ref{thm:1}\ref{11} we use Israel's result to $C^0$-perturb the original roof function into one for which the suspension flow has uncountably many measures of maximal entropy.  With a similar argument we deduce Theorem \ref{thm:1}\ref{12} from Theorem \ref{ru}. Therefore with our method we can prove the density of suspension flows with unique and with uncountably many measures of maximal entropy.

It is worth emphasizing that Theorem \ref{thm:1}\ref{11} holds for a class of dynamical systems that goes far beyond sub-shifts of finite type. Let $X$ be a compact metric space and $T:X\to X$ a continuous map. It follows from our proof that it suffices to verify that Theorem \ref{is} holds for $(X,T)$. This is in fact the case whenever $(X,T)$ has finite topological entropy, the set of ergodic measures is entropy dense and the entropy map is upper semi-continuous; these are the key assumption in the proof of Theorem \ref{is} (for precise definitions and discussion see Remark \ref{rem:entden}).

We finish this paper with a version of Israel's result for suspension flows (see Theorem \ref{thm:issus}).


\section{Preliminaries} \label{pre}

In this section we provide definitions and properties that will be used throughout this paper.

\subsection{Sub-shifts of finite type and thermodynamic formalism}
Let $A$ be a transition matrix defined on the finite alphabet $\mathcal{A}:=\{1, 2, \dots, N\}$. That is, the entries of the matrix $A=A(i,j)_{\mathcal{A} \times \mathcal{A}}$  are zeros and ones.  The one-sided sub-shift of finite type 
 $(\Sigma_A, \sigma)$ associated to $A$ is the set
\begin{equation*}
 \Sigma_A := \left\{ (x_n)_{n \in \N} : A(x_n, x_{n+1})=1  \text{ for every } n \in \N  \right\}, 
\end{equation*} 
 together with the shift map $\sigma: \Sigma_A  \to \Sigma_A $ defined by
$\sigma(x_1, x_2, \dots)=(x_2, x_3,\dots)$. 

We endow $\mathcal{A}$ with the discrete topology and $\mathcal{A}^\N$ with the product topology. On $\Sigma_A$ we consider  the induced topology given by the natural inclusion $\Sigma_A\hookrightarrow \mathcal{A}^\N$. This makes $\Sigma_A$ a compact topological space. The following metric generates this topology, 
\begin{equation*} \label{metric}
d(x,y):=
\begin{cases}
1 & \text{ if } x_1\ne y_1; \\
2^{-k} & \text{ if  } x_i=y_i \text{ for  } i \in \{1, \dots , k\} \text{ and } x_{k+1} \neq y_{k+1}; \\
0 & \text{ if } x=y.
\end{cases}
\end{equation*} 

We say that $(\Sigma_A,\sigma)$ is topologically transitive if there exists a natural number $m \in \N$ such that all the entries of the matrix $A^m$ are positive. Equivalently, there exists a point $x \in \Sigma_A$ having a dense orbit. When the context is clear we will simply write $(\Sigma, \sigma)$. \\

{\bf Standing assumption:} In this paper $(\Sigma,\sigma)$ will always be topologically transitive and the alphabet has cardinality at least $2$. \\ 

Denote by $ \mathcal{M}(\Sigma, \sigma)$ the set of $\sigma-$invariant probability measures and endow it with the weak$^*$ topology.  Recall that a  sequence $(\mu_n)_n$ in $\mathcal{M}(\Sigma, \sigma)$  converges in the weak$^*$ topology to the measure $\mu$ if and only if for every continuous function $f:\Sigma \to \R$ we have that $\lim_{n \to \infty} \int f d \mu_n = \int f d \mu$. The space 
$\mathcal{M}(\Sigma, \sigma)$ is a compact   convex set whose extreme points are the ergodic measures. It is actually a Choquet simplex (each measure is represented in a unique way as a generalized convex combination of the ergodic measures), see \cite[Theorem 6.10 and p.153]{wa} for more details on the structure of $\mathcal{M}(\Sigma, \sigma)$.

Thermodynamic formalism in this setting has been studied, at least,  since the early 1970s.  An account of the theory can be found in \cite[Chapter 9]{wa}. We now recall some of the basic definitions and properties. For every continuous function $\phi:\Sigma \to \R$ the \emph{topological pressure} of $\phi$ is  defined by
\begin{equation*}
P(\phi):= \sup \left\{h(\mu) + \int \phi d \mu : \mu \in \mathcal{M}(\Sigma, \sigma) 			\right\},
\end{equation*}
where $h(\mu)$ denotes the entropy of the measure $\mu$ (see \cite[Chapter 4]{wa} for details). The \emph{topological entropy} of $(\Sigma, \sigma)$ is defined by $h(\sigma):=P(0)$, that is, the pressure of the constant function equal to zero. A measure $\mu \in  \mathcal{M}(\Sigma, \sigma)$ satisfying $P(\phi)= h(\mu) + \int \phi d \mu$, is called an
\emph{equilibrium measure} or \emph{equilibrium state} for $\phi$. If $\phi$ is the constant function equal to zero we call the corresponding equilibrium measures, \emph{measures of maximal entropy}. Since the entropy map $\mu \mapsto h(\mu)$ is upper semi-continuous (see \cite[Theorem 8.2]{wa}) and $\mathcal{M}(\Sigma, \sigma)$ is compact, every continuous function has at least one equilibrium measure.

Denote by $C(\Sigma)$ to the space of continuous real valued functions  on $\Sigma$ endowed  with the uniform norm $\|\cdot\|_0$.  The pressure map
$P(\cdot): C(\Sigma) \to \R$ is continuous and convex functional (see \cite[Theorem 9.7]{wa}).

\begin{remark}
Let $A$ be a transition matrix defined on the finite alphabet $\mathcal{A}$. The two-sided sub-shift of finite type 
 $(\Sigma, \sigma)$ associated to $A$ is the set
\begin{equation*}
 \Sigma := \left\{ (x_n)_{n \in \Z} : A(x_n, x_{n+1})=1  \text{ for every } n \in \Z  \right\}, 
\end{equation*} 
together with the shift map $\sigma: \Sigma  \to \Sigma $ defined by $(\sigma(x))_i=x_{i+1}$, for every $i \in \Z$. All the results discussed in this section remain valid for two-sided sub-shifts of finite type.
\end{remark}

\subsection{Suspension flows} Let $(\Sigma, \sigma)$ be a topologically transitive one-sided sub-shift of finite type and  $\tau \colon \Sigma \to \R^+$ be a positive continuous function. Consider the space
\begin{equation*}\label{shift}
Y_\tau= \{ (x,t)\in \Sigma  \times \R \colon 0 \le t \le\tau(x)\},
\end{equation*}
with the points $(x,\tau(x))$ and $(\sigma(x),0)$ identified for
each $x\in \Sigma $. The \emph{suspension semi-flow} over $\sigma$
with \emph{roof function} $\tau$ is the semi-flow $\Phi_\tau = (
\zeta_t)_{t \ge 0}$ on $Y_\tau$ defined by
\[
 \zeta_t(x,s)= (x,
s+t) \ \text{whenever $s+t\in[0,\tau(x)]$.}
\]
In particular,
\[
 \psi_{\tau(x)}(x,0)= (\sigma(x),0).
\]
We denote such semi-flow by $(Y_\tau, \Phi_{\tau})$. If the context is clear we use the notation $(Y, \Phi)$. If $(\Sigma,\sigma)$ is a two-sided sub-shift of finite type, then $(Y_\tau,\Phi_{\tau})$, where  $\Phi_\tau = (\zeta_t)_{t \in \R}$ is called the \emph{suspension flow} over $\sigma$ with roof function $\tau$; in this case the  acting group is $\R$. For simplicity we state most of our results in Section \ref{sec:mme} in the semi-flow case, but we emphasize they hold in both situations with no distinction (see Remark \ref{rem:onetotwo}).

The natural topology on $Y$ can be obtained by the so called Bowen-Walters metric (see \cite[Section 4]{bw}). Denote by $\M (Y, \Phi)$ the space of $\Phi$-invariant probability measures on $Y$. It is a classical result by Ambrose and Kakutani \cite{ak} that if $\text{Leb}$ denotes the one dimensional Lebesgue measure the map $R: \mathcal{M} (\Sigma, \sigma) \to \M (Y, \Phi)$ defined by 
\begin{equation*} \label{R}
R(\mu):=(\mu \times \text{Leb})|_{Y} /(\mu \times \text{Leb})(Y),
\end{equation*}
is a bijection. A description of the thermodynamic formalism for suspension semi-flows can be found in \cite[Chapter 6]{pp}.   The pressure of a continuous function $g: Y \to \R$ is defined by
\begin{equation*}
P_{\Phi}(g):= \sup \left\{h(\nu) + \int g d \nu : \nu \in \M (Y, \Phi)		\right\}.
\end{equation*}
Again, a measure attaining the supremum in the definition of $P_\Phi(g)$ is called an equilibrium measure for $g$. The \emph{topological entropy} of $(Y, \Phi)$ is defined by $$h(\Phi):=P_{\Phi}(0),$$ that is, the pressure of the constant function equal to zero. A measure $\nu\in \M(Y,\Phi)$ is a \emph{measure of maximal entropy} of $(Y,\Phi)$  if $h(\nu)=h(\Phi)$. 

The entropy of measures in $\M(Y,\Phi)$ can be calculated by means of Abramov's formula (see \cite{a} or \cite[p.91]{pp}):
\[h(\nu)= \frac{h(R^{-1}(\nu))}{ \int \tau dR^{-1}(\nu)},\]
for every $\nu\in \M(Y,\Phi)$.  Given a continuous function $g \colon Y\to\R$ we define the function
$\Delta_g\colon\Sigma\to\R$~by
\[
\Delta_g(x):=\int_{0}^{\tau(x)} g(x,t) \, dt.
\]
The function $\Delta_g$ is also continuous. By Kac' s formula (see \cite[p.90]{pp}) we have
\begin{equation} \label{rela}
\int_{Y} g \, d\nu= \frac{\int_\Sigma \Delta_g\, d
R^{-1}(\nu)}{\int_\Sigma\tau \, d R^{-1}(\nu)},
\end{equation}
for every $\nu\in\M(Y,\Phi)$. Moreover, if $g$ is a H\"older function on $Y$ then $\Delta_g$ is also a H\"older function on $\Sigma$ (see \cite[Proposition 18]{bs}).

 We finish this section with an important observation. 
\begin{remark}\label{rem:root} There is a close relation between the thermodynamics of $(Y,\Phi)$ and that of $(\Sigma , \sigma)$. Indeed, see \cite[Proposition 6.1]{pp}, $P_{\Phi}(g)$ is the unique solution of the equation
\begin{equation*}
P(\Delta_g - t \tau )=0.
\end{equation*}
In particular, the topological entropy of $(Y, \Phi)$ is the unique root of the equation
\begin{equation*}
P(- t \tau )=0,
\end{equation*}
and the measures of maximal entropy for the flow are of the form $R(\mu)$ with $\mu$ equilibrium measure for $-h(\Phi) \tau$.
\end{remark}

\section{Measures of maximal entropy}\label{sec:mme}
In this section we prove the main results of this paper. We first prove the existence of a suspension flow defined over a sub-shift of finite type which admits uncountably many ergodic measures of maximal entropy. For this we can consider an arbitrary sub-shift satisfying our standing assumptions and a suitable roof function. We also prove Theorem \ref{thm:1}, the main result of this paper. Roughly speaking, we prove that any suspension flow defined over a sub-shift of finite type can be slightly perturbed into one with uncountably many ergodic measures of maximal entropy and into one with a unique measure of maximal entropy. This result highlights the differences between the thermodynamic formalism of continuous and H\"older potentials. As mentioned in the introduction, for simplicity we state the result for one-sided sub-shifts and for semi-flows. However, all results are valid for two-sided shifts and suspension flows (see Remark  \ref{rem:onetotwo}).

Let $(\Sigma, \sigma)$ be a topologically transitive one-sided sub-shift of finite type.  The following result is due to Ruelle (see \cite{ru} and \cite[Theorem 3.5]{pp}),

\begin{theorem}[Ruelle] \label{ru}
If $\phi \in C(\Sigma)$ is a H\"older function, then it has a unique equilibrium measure. In particular, the set of functions having unique equilibrium measures is dense in $C(\Sigma)$.
\end{theorem}

\begin{remark}
It is a direct consequence of Theorem \ref{ru} and Remark \ref{rem:root} that if $(Y, \Phi)$ is a suspension flow over $(\Sigma, \sigma)$ with a H\"older continuous roof function, then it has a unique measure of maximal entropy.
\end{remark}

A complementary result by Israel (\cite[Theorem V.2.2]{is})  shows how different thermodynamic formalism is in the space of continuous functions.

\begin{theorem}[Israel] \label{is}
The set of continuous functions having uncountably many ergodic equilibrium measures is dense in $C(\Sigma)$.
\end{theorem}

The following result is a consequence of Israel's theorem.

\begin{proposition}\label{existence1} There exists a suspension flow $(Y_\tau, \Phi_{\tau})$ with continuous roof function $\tau:\Sigma \to \R^+$ which admits uncountably many measures of maximal entropy.
\end{proposition}

\begin{proof}
Since the pressure is continuous the set
\begin{equation*}
\mathcal{L}:=\left\{ \phi \in C(\Sigma) : \sup \phi < P(\phi) \text{ and }\phi > 0	\right\}
\end{equation*}
is  open. Moreover,  for $c >0$ the constant  potential $\phi \equiv c$ belongs to $\mathcal{L}$. Hence,  $\mathcal{L}$ is a non-empty set (by our standing assumption $(\Sigma,\sigma)$ has positive topological entropy). By Theorem \ref{is} there exists $\tau_0 \in \mathcal{L}$ with uncountably  many ergodic equilibrium states. Define
\begin{equation*}
\tau:= P(\tau_0)-\tau_0.
\end{equation*}
Then $\tau >0$, $P(-\tau)=P(\tau_0) - P(\tau_0)=0$ and $-\tau$ has uncountably many equilibrium states. In light of Remark \ref{rem:root} we can conclude that the the suspension flow
$(Y_\tau, \Phi_{\tau})$ has uncountably many ergodic measures of maximal entropy.
\end{proof}

Our next result establishes that with a small continuous time  reparametrization of the flow we can obtain a suspension flow with uncountably many measures of maximal entropy. Let $(Y_{\tau_1}, \Phi_{\tau_1})$ and $(Y_{\tau_2}, \Phi_{\tau_2})$ be two suspension semi-flows with corresponding continuous roof functions  $\tau_1: \Sigma \to \R$ and $\tau_2: \Sigma \to \R$. Then the semi-flow $(Y_{\tau_1}, \Phi_{\tau_1})$ is a time  reparametrization  of  $(Y_{\tau_2}, \Phi_{\tau_2})$. Indeed,  the map $\pi:Y_1 \to Y_2$ defined by
\begin{equation*}
\pi(x,s)=\left(x ,  \frac{\tau_2(x)}{\tau_1(x)} s  \right),
\end{equation*}
preserves the orbit structure. It actually send leafs to leafs and corresponds to the time change. A natural way of estimating the size of the time  reparametrization is by $\| \tau_1 - \tau_2\|_0$ or equivalently, since $\tau_1$ is bounded below, by $\| \tau_2/ \tau_1  - 1 \|_0$ (see \cite{ci} for a related discussion). 

%

\begin{proof}[Proof of Theorem \ref{thm:1}]
We divide the proof in two Lemmas.

\begin{lemma}\label{lem1} Let $\psi\in C(\Sigma)$ be such that $P(-\psi)=0$. Then there exist two sequences  of continuous potentials $(\phi_n)_{n\in \N}$ and $(\rho_n)_{n \in\N}$ such that 
\begin{enumerate}
\item Both sequences converge uniformly to $\psi$, $\lim_{n\to \infty} \phi_n=\psi$ and $\lim_{n\to \infty} \rho _n=\psi$.
\item For every $n \in \N$ we have $P(-\phi_n)=0$ and $P(-\rho_n)=0.$
\item For every $n \in \N$ the potential $-\phi_n$ has infinitely many ergodic equilibrium states.
\item For every $n \in \N$ the potential $-\rho_n$ has a unique equilibrium state.
\end{enumerate}
\end{lemma}

\begin{proof} By Theorem \ref{is} there exists a sequence of continuous functions $(f_n)_n$ converging uniformly to $-\psi$ such that for every $n \in \N$ the function $f_n$ has uncountably many ergodic equilibrium measures. For every $n \in \N$ define
$\phi_n=P(f_n)-f_n$. Since the pressure is a continuous function and $P(-\psi)=0$ we have that 
\begin{equation*}
\lim_{n\to\infty}P(f_n)=0\text{ and }  \lim_{n\to \infty} \phi_n=\psi.
\end{equation*}
Moreover, each function $-\phi_n$ has uncountably many ergodic equilibrium measures and $P(-\phi_n)=0$. Analogously, by Theorem \ref{ru} there exists a sequence of continuous functions $(h_n)_n$ converging uniformly to $-\psi$ such that for every $n \in \N$ the function $h_n$ has a unique equilibrium measure. For every $n \in \N$ define $\rho_n=P(h_n)-h_n$. It follows from the continuity of the pressure and $P(-\psi)=0$ that
\begin{equation*}
\lim_{n\to\infty}P(h_n)=0\text{ and }  \lim_{n\to \infty} \rho_n=\psi.
\end{equation*}
Note that $P(-\rho_n)=0$. Moreover, every function $\rho_n$ has a unique equilibrium measure. 
\end{proof}

\begin{lemma} \label{density} Let $\tau: \Sigma \to \R$ be a positive continuous potential. Then there exist two sequences $(\tau_n)_{n\in\N}$  and $(\tau'_n)_{n\in \N}$  of continuous positive functions such that 
\begin{enumerate}
\item The sequences $(\tau_n)_{n\in \N}$ and $(\tau'_n)_{n\in\N}$ converge uniformly to $\tau$.
\item For every $n \in \N$ the suspension semi-flow with roof function $\tau_n$ has uncountably many measures of maximal entropy.
\item For every $n \in \N$ the suspension semi-flow with roof function $\tau'_n$ has a unique measure of maximal entropy.
\end{enumerate} 
\end{lemma}
\begin{proof} 
In order to prove this result we will construct a sequence of positive continuous functions $(\tau_n)_n$ converging uniformly to $\tau$ such that $P(- h(\Phi_{\tau}) \tau_n)=0$ and $- h(\Phi_{\tau}) \tau_n$ has uncountably many ergodic measures of maximal entropy. 
By Lemma \ref{lem1} there exits a sequence  of continuous functions $(\phi_n)_n$ converging uniformly to $h(\Phi_{\tau}) \tau$, such that for every $n \in \N$ we have that $P(-\phi_n)=0$ and $\phi_n$ has uncountably many ergodic equilibrium measures. Define the functions $\tau_n: \Sigma \to \R$ by
\begin{equation*}
\tau_n(x):= \tau(x) + \frac{\phi_n(x) - h(\Phi_{\tau}) \tau(x)}{ h(\Phi_{\tau})}.
\end{equation*}
Note that $(\tau_n)_n$ converges uniformly to $\tau$. Moreover, 
\begin{equation*}
- h(\Phi_{\tau}) \tau_n(x)= - h(\Phi_{\tau}) \tau(x) - \phi_n(x) + h(\Phi_{\tau}) \tau(x)= - \phi_n(x).
\end{equation*}
Thus, $P(- h(\Phi_{\tau}) \tau_n)=0$ and $- h(\Phi_{\tau}) \tau_n$ has uncountably many ergodic equilibrium measures. 
The construction of the sequence $(\tau'_n)_{n\in\N}$ is done analogously.
\end{proof}
The result follows considering the suspension semi-flows $(Y_{\tau'_m}, \Phi_{\tau_m})$ and $(Y_{\tau'_m}, \Phi_{\tau'_m})$ for $m \in \N$ sufficiently large.
\end{proof}

\begin{remark}
Note that the suspension semi-flows obtained by time reparametrization, $(Y_{\tau_n}, \Phi_{\tau_n})$ ,  $(Y_{\tau'_n}, \Phi_{\tau'_n})$ constructed in the proof of Theorem \ref{thm:1} have the same topological entropy as the original semi-flow $(Y_{\tau}, \Phi_{\tau})$.
\end{remark}

\begin{remark}\label{rem:onetotwo} 
All the results in this article, in particular Theorem \ref{thm:1}, also hold in the context of two-sided sub-shifts and suspension flows. For this it suffices to notice that Theorem \ref{is} and Theorem \ref{ru} also hold for two-sided sub-shifts. Indeed, a well known result by Sinai \cite[Proposition 1.2]{pp} allows to reduce the thermodynamic formalism of H\"older functions from two sided sub-shifts of finite type to one-sided ones.  Sinai proved that any H\"older function defined on a two-sided sub-shift is cohomologous a function that depends only on future coordinates. This result was later generalized by Walters \cite[Theorem 4(i)]{wa2} who found optimal regularity assumptions on the functions for the above property to hold. This readily implies that Theorem \ref{ru} also hold for two-sided sub-shifts.  On the other hand, Theorem \ref{is} can be directly proven  in the context of two-sided sub-shifts of finite type (see \cite[Theorem V.2.2]{is}). Moreover, the two-sided version of Israel's Theorem can also be deduced in a similar fashion as Ruelle's Theorem by means of a result by Walters (see \cite[Theorem 4 (iii)]{wa2}). In that theorem it is proven that for every  continuous function  on the two-sided shift there exists a function that depends only in  future coordinates, so that the difference of both functions belongs to the closure of the set of coboundaries. 
\end{remark}

\begin{remark}\label{rem:entden} As mentioned in the introduction, Theorem \ref{thm:1}\ref{11} holds for a class of base dynamical systems that goes far beyond sub-shifts of finite type. We proceed to explain the relevant assumptions here.  Let $X$ be a compact metric space and $T:X\to X$ a continuous map. Denote by  $\M_T$  the space of invariant probability measures of $(X,T)$ and by $\cE_T$ the subset of ergodic ones. 

\begin{definition}We say that $\cE_T$ is \emph{entropy dense} in $\M_T$ if for every $\mu\in \M_T$, there exists a sequence $(\mu_n)_n$ in $\E_T$ that converges in the weak$^*$ topology to $\mu$ and $\lim_{n\to\infty}h(\mu_n)=h(\mu)$. 
\end{definition}

\begin{definition} We say that the entropy map of $(X,T)$ is upper semi-continuous if for every sequence $(\mu_n)_n$ which converges in the weak$^*$ topology to $\mu$, then $\limsup_{n\to\infty}h(\mu_n)\le h(\mu)$. \end{definition}
 
It follows directly from the proof of \cite[Theorem V.2.2]{is} that if the topological entropy of $(X,T)$ is finite, $\cE_T$ is entropy dense in $\M_T$ and the entropy map of $(X,T)$ is upper semi-continuous, then Theorem \ref{is} holds for $(X,T)$. In this situation, the exact same proof of Theorem \ref{thm:1}\ref{11} yields  the density of suspension flows over $(X,T)$ with uncountably many ergodic measures of maximal entropy. We remark that these three properties hold for sub-shifts of finite type.

Regarding these assumptions, the finite entropy and the upper semi-continuity of the entropy map are crucial in order to identify tangent functionals to the pressure at $F:X\to \R$ with equilibrium states of $F$ (see \cite[Theorem 9.12]{wa}). This allows  to incorporate tools from functional analysis into the problem. The entropy density of $\cE_T$ is important to ensure that the sets  
$$A(F,\epsilon):= \left\{\mu\in\cE_T: P(F)-\epsilon<h(\mu)+\int Fd\mu \right\},$$
support a non-atomic probability measure, say $\nu_\epsilon$, for every $\epsilon>0$.  This is the fact that allows Israel to construct potentials with uncountably many ergodic equilibrium states. These are ergodic measures which lie in the support of $\nu_\epsilon$.

\end{remark}

\begin{remark}\label{rem:finn}
The proof of Theorem \ref{thm:1}\ref{12} is a consequence of Theorem \ref{ru} and the density of H\"older continuous potentials in the space of continuous functions (we always assume our phase space to be a compact metric space). Theorem \ref{ru} was generalized by Bowen \cite{bo2}, who proved that if $(X,T)$ is expansive and has the specification property, then every H\"older continuous potential has a unique equilibrium state. In particular, Theorem \ref{thm:1}\ref{12} holds under those assumptions on the base dynamics. 
\end{remark}

%
%

\subsection{Equilibrium measures for suspension semi-flows and flows} In this sub-section we make use of the previous results to show that in the context of suspension semi-flows and flows over sub-shifts of finite type, the set of potentials having uncountably many ergodic equilibrium measures is dense in the space of continuous functions $C(Y)$. This is the analog of Theorem \ref{is} for the suspension flow. As before, we state the result for suspension semi-flows but the same proof works in the two-sided case. 
%
%
%
%
%
%
%

\begin{theorem}\label{thm:issus}
Let $(Y, \Phi)$ be the suspension semi-flow over a one-sided sub-shift of finite type $(\Sigma, \sigma)$ with roof function $\tau: \Sigma \to \R$ and a continuous function
$g:Y  \to \R$. Given $\epsilon >0$ there exists a continuous function $h: Y \to \R$ having  uncountably many ergodic equilibrium measures and $\max_{(x,t)\in Y}|g(x,t)-h(x,t)|<\epsilon$.
 \end{theorem}

\begin{proof} It will be convenient to consider $C>0$ such that $g_0:=g+C$ is strictly positive. In particular we have that $\Delta_{g_0}(x)>0$, for every $x\in\Sigma$.  By Lemma \ref{lem1} there exists a sequence of continuous functions $(\phi_n)_n$ with $\phi_n:\Sigma \to \R$ such that for every $n \in \N$ we have $P(-\phi_n)=0$, the function $-\phi_n$ has uncountably many ergodic equilibrium measures and
$\phi_n$ converges uniformly to $P_{\Phi}(g_0)\tau -\Delta_{g_0}$. For every $n \in \N$ define $F_n:\Sigma \to \R$ by
\begin{equation*}
F_n(x):= P_{\Phi}(g_0) \tau(x) - \phi_n(x).
\end{equation*}
For every $n \in \N$, define the function $g_n(x,t):Y \to \R$ by
\begin{equation*}
g_n(x,t):=\frac{F_n(x)}{\Delta_{g_0}(x)}g_0(x,t).
\end{equation*}
Note that  $F_n(x)=\Delta_{g_n}(x)$, indeed
\begin{equation*}
\Delta_{g_n}(x)= \int_0^{\tau(x)}g_n(x,t)dt=\frac{F_n(x)}{\Delta_{g_0}(x)}\int_0^{\tau(x)}g_0(x,t)dt=F_n(x).
\end{equation*}
Thus, $\Delta_{g_n}= P_{\Phi}(g_0)\tau -\phi_n$ and in particular $(\Delta_{g_n})_n
$ converges uniformly to $\Delta_{g_0}$.
Therefore, $P(\Delta_{g_n} -P_{\Phi}(g_0)\tau)=0$ which implies that $P_{\Phi}(g_n)= P_{\Phi}(g_0)$ and that $g_n$ has uncountably many equilibrium measures. For $n$ sufficiently large we have that 
$$\|\Delta_{g_n}-\Delta_{g_0}\|_0 <  \frac{\epsilon\inf_{x\in\Sigma}\Delta_{g_0}(x)}{\sup_{(x,t)\in Y}g_0(x,t)},$$ which implies that   $$|g_n(x,t)-g_0(x,t)|= g_0(x,t) \left|\frac{\Delta_{g_n}(x)}{\Delta_{g_0}(x)}-1 \right|< \epsilon.$$
We conclude that $\max_{(x,t)\in Y}|g_n(x,t)-g_0(x,t)|<\epsilon$, for every $n$ sufficiently large. Finally, consider  $n \in \N$ large enough and set $h:=g_n-C$. 

%
%

\end{proof}

\end{document}